\documentclass[a4paper,11pt,reqno]{article}

\usepackage{amsmath,amssymb,eucal,amsthm}
\usepackage{latexsym,bm}
\usepackage{mathrsfs,amsmath,amssymb,secdot}
\usepackage{amscd}

\allowdisplaybreaks

\DeclareFontEncoding{OT2}{}{}
\DeclareFontSubstitution{OT2}{cmr}{m}{n}
\DeclareSymbolFont{cyss}{OT2}{wncyss}{m}{n}
\DeclareMathSymbol{\sh}{\mathbin}{cyss}{`x}

\newtheorem{theorem}{Theorem}[section]
\newtheorem{lemma}[theorem]{Lemma}
\newtheorem{prop}[theorem]{Proposition}
\newtheorem{corollary}[theorem]{Corollary}

\theoremstyle{definition}

\newtheorem{remark}[theorem]{Remark}

\DeclareFontEncoding{OT2}{}{}
\DeclareFontSubstitution{OT2}{cmr}{m}{n}
\DeclareSymbolFont{cyss}{OT2}{wncyss}{m}{n}
\DeclareMathSymbol{\sh}{\mathbin}{cyss}{`x}

\numberwithin{equation}{section}
\allowdisplaybreaks

\begin{document}

\title{On a duality formula for certain sums of values of poly-Bernoulli polynomials and its application}

\author{Masanobu Kaneko, Fumi Sakurai, and Hirofumi Tsumura}

\date{\today}

\maketitle

\begin{abstract}
  We prove a duality formula for certain sums of values of poly-Bernoulli polynomials which generalizes dualities for poly-Bernoulli numbers. 
  We first compute two types of generating functions for these sums, from which  the duality formula is apparent. 
  Secondly we give an analytic proof of the duality from the viewpoint of our previous study of zeta-functions
  of Arakawa-Kaneko type. As an application, we give a formula that relates poly-Bernoulli numbers to the Genocchi numbers.
\end{abstract}

%%%%%%%%%%%%%%%%%%%%%%%%%%%%%%%%%%%%%%%%%%%%                                   
\section{Introduction} \label{sec-1}
%%%%%%%%%%%%%%%%%%%%%%%%%%%%%%%%%%%%%%%%%%%%  

Two types of poly-Bernoulli numbers $\{B_n^{(k)}\}$ and $\{C_n^{(k)}\}$ are defined by 
the generating series
\begin{equation}
\begin{split}
&\frac{{\rm Li}_{k}(1-e^{-t})}{1-e^{-t}}=\sum_{n=0}^\infty B_n^{(k)}\frac{t^n}{n!}, \\
&\frac{{\rm Li}_{k}(1-e^{-t})}{e^t-1}=\sum_{n=0}^\infty C_n^{(k)}\frac{t^n}{n!} 
\end{split}
\label{e-1-1}
\end{equation}
for $k\in \mathbb{Z}$, where ${\rm Li}_{k}(z)$ is the polylogarithm function given by
\begin{equation}
{\rm Li}_{k}(z)=\sum_{m=1}^\infty \frac{z^m}{m^k}\quad (|z|<1) \label{e-1-3}
\end{equation}
(see Kaneko \cite{Kaneko1997} and Arakawa-Kaneko \cite{AK1999}, also Arakawa-Ibukiyama-Kaneko \cite{AIK2014}). 
Noting ${\rm Li}_1(z)=-\log(1-z)$, we see that $C_n^{(1)}$ coincides with the ordinary Bernoulli number $B_n$ defined by
\begin{equation*}
\frac{t}{e^t-1}=\sum_{n=0}^\infty \,B_n\frac{t^n}{n!}, 
\end{equation*}
and that $B_n^{(1)}=B_n$ for $n \in \mathbb{Z}_{\geq 0}$ with $n\neq 1$.

These numbers have been actively investigated and many interesting properties and formulas for them 
have been discovered (see, for example, \cite{BH2015,Brew2008,CGS2014,HM2007-1,HM2007-2,Kaneko-Mem}). 
Of them we highlight the following duality formulas obtained by the first-named author:
\begin{align}
& B_m^{(-l)}=B_{l}^{(-m)},\label{e-1-5}\\
& C_m^{(-l-1)}=C_{l}^{(-m-1)} \label{e-1-5-2}
\end{align}
for any $l,m\in \mathbb{Z}_{\geq 0}$, which can be shown by considering their generating functions 
in two variables (see \cite[Theorem 2]{Kaneko1997} and \cite[\S\,2]{Kaneko-Mem}).

The poly-Bernoulli polynomials are defined by
\begin{align}
&e^{-xt}\frac{{\rm Li}_{k}(1-e^{-t})}{1-e^{-t}}=\sum_{n=0}^\infty B_n^{(k)}(x)\frac{t^n}{n!}  \label{e-1-6}
\end{align}
(see Coppo-Candelpergher \cite{CC2010}). It can be easily checked that
\begin{align*}
& B_n^{(k)}(0)=B_n^{(k)},\quad B_n^{(k)}(1)=C_n^{(k)}, \\
& B_n^{(k)}(x)=\sum_{j=0}^{n}(-1)^{n-j}\binom{n}{j}B_j^{(k)}x^{n-j}. 
\end{align*}

The main purpose of this paper is to generalize the duality formulas \eqref{e-1-5} and \eqref{e-1-5-2} as follows:
the identity
\begin{equation}
\sum_{j=0}^{n} {n \brack j} B_{m}^{(-l-j)}(n) = \sum_{j=0}^{n} {n \brack j} B_{l}^{(-m-j)}(n) \label{e-1-9}
\end{equation}
holds for any $l , m, n\in \mathbb{Z}_{\geq 0}$ (see Corollary \ref{Cor-1}), where 
$\{{n \brack j} \mid n,j\in \mathbb{Z}_{\geq 0}\}$ are the Stirling numbers of the first kind (for the definition, see \eqref{Stir}). 
In particular, we easily see that \eqref{e-1-9} for the cases of $n=0$ and $n=1$ coincide with \eqref{e-1-5} and \eqref{e-1-5-2}, respectively. 
Hence \eqref{e-1-9} can be regarded as a ``one-parameter'' generalization of the duality formula for poly-Bernoulli numbers.
It is an interesting question whether this generalization also has some nice combinatorial interpretations like
those described in \cite{BH2015,Brew2008,CGS2014}.

In Section \ref{sec-2}, we give an elementary proof of \eqref{e-1-9}. In fact, denoting the left-hand side of \eqref{e-1-9} 
by $\mathscr{B}_{m}^{(-l)}(n)$, we calculate two types of generating functions of 
$\{\mathscr{B}_{m}^{(-l)}(n)\}_{l,m\geq 0}$ in two variables (see Theorem \ref{Th-main-1}), 
which turn out to be symmetric and hence  \eqref{e-1-9} follows. In Section \ref{sec-3}, 
we give an analytic proof \eqref{e-1-9} from the viewpoint of our previous study of zeta-functions of Arakawa-Kaneko type. 
The method is similar to that used by the first-named and the third-named authors in \cite{Kaneko-Tsumura2015}. In 
the final Section \ref{sec-4}, as an application of Theorem \ref{Th-main-1}, we prove the relation
$$
\sum_{l=0}^{n}(-1)^{l}C_{n-l}^{(-l-1)} = -G_{n+2}\quad (n\in \mathbb{Z}_{\geq 0})
$$
(see Theorem \ref{Th-main-2}), where $G_{n} = (2-2^{n+1})B_{n}$ $(n\in \mathbb{Z}_{\geq 0})$ is the Genocchi number (see, for example, Lucas \cite[P.\,250]{Lucas}, also Stanley \cite[Exercise 5.8]{Stanley}). This can be regarded as a ``$C$-version'' of the known formula for $B_m^{(-l)}$ (see \cite[Proposition]{AK-R1999}):
$$
\sum_{l=0}^{n}(-1)^{l}B_{n-l}^{(-l)} = 0 \quad (n\in \mathbb{Z}_{\geq 1}).
$$
\ 

%%%%%%%%%%%%%%%%%%%%%%%%%%%%%%%%%%%%%%%%%%%%%%%%%%%%%%%%%%%%%%
\section{A generalization of the duality formula}\label{sec-2}
%%%%%%%%%%%%%%%%%%%%%%%%%%%%%%%%%%%%%%%%%%%%%%%%%%%%%%%%%%%%%%

Let ${n \brack m}$ and ${n \brace m}$ $(n,m\in \mathbb{Z}_{\geq 0})$ be the Stirling numbers 
of the first and the second kind determined respectively by the following recursion relations:
\begin{equation}
\begin{split}
& {0 \brack 0}=1,\quad {n \brack 0}={0 \brack m}=0\ \ (m,n\neq 0),\\
& {n+1 \brack m}={n \brack m-1}+n{n \brack m}\ \ (n\geq 0,\ m\geq 1),
\end{split}
\label{Stir}
\end{equation}
and 
\begin{align*}
& {0 \brace 0}=1,\quad {n \brace 0}={0 \brace m}=0\ \ (m,n\neq 0),\\
& {n+1 \brace m}={n \brace m-1}+m{n \brace m}\ \ (n\geq 0,\ m\geq 1),
\end{align*}
(see for example, \cite[Definitions 2.2 and 2.4]{AIK2014}). 

As mentioned in Section \ref{sec-1}, we let 
\begin{equation}
\mathscr{B}_{m}^{(-l)}(n):=\sum_{j=0}^{n} {n \brack j} B_{m}^{(-l-j)}(n) \label{e-2-1}
\end{equation}
for $ l , m, n\in \mathbb{Z}_{\geq 0}$. Note that 
$$\mathscr{B}_{m}^{(-l)}(0)={B}_{m}^{(-l)},\quad \mathscr{B}_{m}^{(-l)}(1)={C}_{m}^{(-l-1)}.$$
The first main result of this paper is the following theorem. 

\begin{theorem}\label{Th-main-1} For $n\in \mathbb{Z}_{\geq 0}$, we have
\begin{align}
& \sum_{l=0}^{\infty}\sum_{m=0}^{\infty}\mathscr{B}_{m}^{(-l)}(n)\frac{x^{l}}{l!}\frac{y^{m}}{m!} 
=  \frac{n!\,e^{x+y}}{(e^{x}+e^{y}-e^{x+y})^{n+1}} \label{e-2-2}\\
\intertext{and}
& \sum_{l=0}^{\infty}\sum_{m=0}^{\infty}\mathscr{B}_{m}^{(-l)}(n)x^{l}y^{m} 
= \sum_{j=0}^{\infty} j!\ (j+n)!\ Q_{j}(x)Q_{j}(y), \label{e-2-3}
\end{align}
where 
$$Q_{j}(X) = \frac{X^{j}}{(1-X)(1-2X) \cdots (1-(j+1)X)}\quad (j\in \mathbb{Z}_{\geq 0}).$$
\end{theorem}

From \eqref{e-2-2} or \eqref{e-2-3}, we immediately obtain the following result which contains \eqref{e-1-5} 
and \eqref{e-1-5-2} as the special cases $n=0,1$.

\begin{corollary}\label{Cor-1}\ \ For $ l , m, n\in \mathbb{Z}_{\geq 0}$, it holds  
$\mathscr{B}_{m}^{(-l)}(n)=\mathscr{B}_{l}^{(-m)}(n)$, namely
\begin{equation}
\sum_{j=0}^{n} {n \brack j} B_{m}^{(-l-j)}(n) = \sum_{j=0}^{n} {n \brack j} B_{l}^{(-m-j)}(n). \label{e-2-4}
\end{equation}
\end{corollary}

To prove Theorem \ref{Th-main-1}, we start with the following lemma.

\begin{lemma}[Takeda \cite{Takeda}]\label{lem-Takeda}\ \ 
For $n,r\in \mathbb{Z}_{\geq 0}$ with $r\ge n$, 
\begin{align}
&\frac{e^{nt}(e^{t}-1)^{r-n}}{(r-n)!} 
= \sum_{m=0}^{\infty}\sum_{i=0}^{n}(-1)^{n-i}{n \brack i}{m+i \brace r}\frac{t^{m}}{m!}, \label{e-2-5}\\
& \sum_{i=0}^{n}{n \brace i}\frac{e^{it}(e^{t}-1)^{r-i}}{(r-i)!}=\sum_{m=0}^{\infty}{m+n \brace r}\frac{t^{m}}{m!}. \label{e-2-6}
\end{align}
\end{lemma}

\begin{proof}
We sketch the proof of this lemma. 

As for \eqref{e-2-5}, we use the induction on $n\geq 0$. 
The case of $n=0$ reduces to the well-known identity
\[ \frac{(e^{t}-1)^{m}}{m!} = \sum_{n=m}^{\infty}{n \brace m}\frac{t^{n}}{n!}.\]
(See for instance \cite[Proposition 2.6]{AIK2014}.) Assume \eqref{e-2-5} for the case of $n$, and compute its derivative. 
Then, for $r\geq n+1$, we have
\begin{equation*}
\frac{ne^{nt}(e^{t}-1)^{r-n}}{(r-n)!} + \frac{e^{(n+1)t}(e^{t}-1)^{r-n-1}}{(r-n-1)!}=\sum_{m=0}^{\infty}
\sum_{i=0}^{n}(-1)^{n-i}{n \brack i}{m+i+1 \brace r}\frac{t^{m}}{m!}.
\end{equation*}
By the induction hypothesis, we obtain
\begin{align*} \nonumber
& \lefteqn{\frac{e^{(n+1)t}(e^{t}-1)^{r-n-1}}{(r-n-1)!}} \\ \nonumber
& =  \sum_{m=0}^{\infty}\sum_{i=0}^{n}(-1)^{n-i}{n \brack i}{m+i+1 \brace r}\frac{t^{m}}{m!}
- \frac{ne^{nt}(e^{t}-1)^{r-n}}{(r-n)!} \\ \nonumber
& =  \sum_{m=0}^{\infty}\sum_{i=1}^{n+1}(-1)^{n-i+1}{n \brack i-1}{m+i \brace r}\frac{t^{m}}{m!}
- n\sum_{m=0}^{\infty}\sum_{i=0}^{n}(-1)^{n-i}{n \brack i}{m+i \brace r}\frac{t^{m}}{m!} \\ 
& =  \sum_{m=0}^{\infty}\sum_{i=1}^{n+1}(-1)^{n-i+1}\left({n \brack i-1}+n{n \brack i}\right)
{m+i \brace r}\frac{t^{m}}{m!}\\
& =  \sum_{m=0}^{\infty}\sum_{i=0}^{n+1}(-1)^{n-i+1}{n+1 \brack i}{m+i \brace r}\frac{t^{m}}{m!}.
\end{align*}
Therefore we complete the proof of \eqref{e-2-5}.

As for \eqref{e-2-6}, similar to the above proof, 
considering the derivative of \eqref{e-2-6}, we inductively obtain the assertion. 
\end{proof}

Next we show the following proposition which is a certain generalization of the known result for 
ordinary poly-Bernoulli numbers given by the first-named author \cite[Theorem 1]{Kaneko1997}. 

\begin{prop}\label{prop-2-5} \ \ 
For $m,n\in \mathbb{Z}_{\geq 0}$ and $k\in \mathbb{Z}$, 
\begin{equation}
B_{m}^{(k)}(n) =  \sum_{q=1}^{m+1}\sum_{i=0}^{n}(-1)^{m+n+q-i-1}\frac{(q-1)!}{q^{k}}
{n \brack i}{m+i \brace n+q-1}. \label{e-2-7}
\end{equation}
\end{prop}

\begin{proof}
By definition, we have 
\begin{align*} 
e^{-nt}\frac{{\rm Li}_{k}(1-e^{-t})}{1-e^{-t}}
& =  e^{-nt}\sum_{q=1}^{\infty}(-1)^{q-1}\frac{(e^{-t}-1)^{q-1}}{q^{k}}.
\end{align*}
Using \eqref{e-2-5} with $r\to q-1 $ and $t\to -t$ on the right, we obtain
\begin{align*}
e^{-nt}\frac{{\rm Li}_{k}(1-e^{-t})}{1-e^{-t}}
& =  \sum_{m=0}^{\infty}\sum_{q=1}^{\infty}\sum_{i=0}^{n}(-1)^{m+n+q-i-1}\frac{(q-1)!}{q^{k}}
{n \brack i}{m+i \brace n+q-1}\frac{t^{m}}{m!}.
\end{align*}
Comparing the coefficients of $t^m/m!$ on both sides and noting ${m+i \brace n+q-1}=0$ when $q>m+1$, we complete the proof.
\end{proof}

Now we give the proof of Theorem \ref{Th-main-1}. 

\begin{proof}[Proof of Theorem \ref{Th-main-1}] \ 
First we will prove \eqref{e-2-2}. 
Substituting \eqref{e-2-7} with $k\to -l-j$ into \eqref{e-2-1}, we obtain 
\begin{align*}
\mathscr{B}_m^{(-l)}(n)
& =  \sum_{j=0}^{n}{n \brack j}\sum_{q=1}^{m+1}\sum_{i=0}^{n}(-1)^{m+n+q-i-1}(q-1)!q^{l+j}
{n \brack i}{m+i \brace n+q-1}.
\end{align*}
With this we compute the generating function
\begin{align*} \nonumber
& F_n(x,y):=\lefteqn{\sum_{l=0}^{\infty}\sum_{m=0}^{\infty}\mathscr{B}_m^{(-l)}(n)\frac{x^{l}}{l!}\frac{y^{m}}{m!}} \\ \nonumber
& \ =  \sum_{l=0}^{\infty}\sum_{m=0}^{\infty}\left(
\sum_{j=0}^{n}{n \brack j}\sum_{q=1}^{m+1}\sum_{i=0}^{n}(-1)^{m+n+q-i-1}(q-1)!\ q^{l+j}
{n \brack i}{m+i \brace n+q-1}\right)\frac{x^{l}}{l!}\frac{y^{m}}{m!} \\ 
& \ =  \sum_{m=0}^{\infty}\sum_{q=1}^{m+1}(-1)^{m}
\sum_{j=0}^{n}{n \brack j}q^{j}\sum_{i=0}^{n}(-1)^{n+q-i-1}e^{qx}(q-1)!
{n \brack i}{m+i \brace n+q-1}\frac{y^{m}}{m!}. 
\end{align*}
By the well-known identity $(x)_n:=x(x+1)\cdots(x+n-1)=\sum_{j=0}^n{n \brack j}x^j$,  this is equal to 
\begin{align*} \nonumber
& \sum_{m=0}^{\infty}\sum_{q=1}^{m+1}(-1)^{m}
\sum_{i=0}^{n}(-1)^{n+q-i-1}e^{qx}(q-1)!(q)_n
{n \brack i}{m+i \brace n+q-1}\frac{y^{m}}{m!} \\ \nonumber
& =  \sum_{r=0}^{\infty}
\sum_{i=0}^{n}(-1)^{n+r-i}e^{(r+1)x}r!(r+1)_n
{n \brack i} \sum_{m=r}^{\infty}{m+i \brace n+r}\frac{(-y)^{m}}{m!}.
\end{align*}
Note that $m$ may run over all non-negative integers in the last sum because ${m+i \brace n+r}=0$ for $0\leq m\leq r-1$. 
Hence, by \eqref{e-2-6} and the formula $\sum_{l=0}^\infty (-1)^{l}{n \brack l}{l \brace m} = (-1)^{n}\delta_{m, n}$ 
($\delta_{m, n}$ is the Kronecker delta,
see \cite[Proposition 2.6]{AIK2014}), we have 
\begin{align*} \nonumber
 F_n(x,y)  &  =  \sum_{r=0}^{\infty}\sum_{i=0}^{n}\sum_{g=0}^{i}
(-1)^{n+r-i}e^{(r+1)x}r!(r+1)_n{n \brack i}
\frac{{i \brace g}e^{-gy}(e^{-y}-1)^{n+r-g}}{(n+r-g)!} \\ \nonumber
& =  \sum_{r=0}^{\infty}(-1)^{n+r}e^{(r+1)x}r!(r+1)_n
\sum_{g=0}^{n}\sum_{i=g}^{n}(-1)^{i}{n \brack i}{i \brace g}
\frac{e^{-gy}(e^{-y}-1)^{n+r-g}}{(n+r-g)!} \\
& =  \sum_{r=0}^{\infty}(-1)^{n+r}e^{(r+1)x}r!(r+1)_n
\sum_{g=0}^{n}(-1)^{n}\delta_{n, g}\frac{e^{-gy}(e^{-y}-1)^{n+r-g}}{(n+r-g)!}\\
& =  e^{-ny}e^{x}\sum_{r=0}^{\infty}e^{rx}(1-e^{-y})^{r}(r+1)_n \\
& =  e^{-ny}e^{x}n!\sum_{r=0}^{\infty}(e^{x}-e^{x-y})^{r}{n+r \choose n} \\
& =  \frac{n!e^{x+y}}{(e^{x}+e^{y}-e^{x+y})^{n+1}}.
\end{align*}
This completes the proof of \eqref{e-2-2}.

Next we will prove \eqref{e-2-3}. From \eqref{e-2-2}, we have
\begin{align*}
F_n(x,y) & =  \frac{n! e^{x+y}}{\{ 1-(e^{x}-1)(e^{y}-1) \}^{n+1}} \\
& =  n!\ e^{x+y}\sum_{j=0}^{\infty}{j+n \choose n}(e^{x}-1)^{j}(e^{y}-1)^{j} \\
& =  n! \sum_{j=0}^{\infty}{j+n \choose n}\frac{1}{(j+1)^{2}}
           \frac{d}{dx}(e^{x}-1)^{j+1}\frac{d}{dy}(e^{y}-1)^{j+1} \\
& =  n! \sum_{j=0}^{\infty}\sum_{l=j}^{\infty}\sum_{m=j}^{\infty}(j!)^{2}
          {j+n \choose n}{l+1 \brace j+1}{m+1 \brace j+1}\frac{x^{l}}{l!}\frac{y^{m}}{m!}.
\end{align*}
Hence, noting ${l+1 \brace j+1}{m+1 \brace j+1} = 0$ for $j > \min(l, m)$, we obtain
\begin{equation}
\mathscr{B}_m^{(-l)}(n) 
= \sum_{j=0}^{\min(l, m)}n!\ (j!)^{2}{j+n \choose n}{l+1 \brace j+1}{m+1 \brace j+1}. \label{e-2-9}
\end{equation}
By the identity (see \cite[Proposition 2.6]{AIK2014})
\[ Q_{m}(t)=\frac{t^{m+1}}{(1-t)(1-2t) \cdots (1-(m+1)t)} = \sum_{n=m+1}^{\infty}{n \brace m+1}t^{n}\]
and \eqref{e-2-9}, we have
\begin{align*}
\sum_{j=0}^{\infty}j!\ (j+n)!\ Q_{j}(x)Q_{j}(y)
& =  \sum_{j=0}^{\infty}j!\ (j+n)!\ \sum_{l=j}^{\infty}{l+1 \brace j+1}x^{l}
           \sum_{m=j}^{\infty}{m+1 \brace j+1}y^{m} \\
& =  \sum_{l=0}^{\infty}\sum_{m=0}^{\infty}\sum_{j=0}^{\min(l, m)}
           n!\ (j!)^{2}{j+n \choose n}{l+1 \brace j+1}{m+1 \brace j+1}x^{l}y^{m} \\
& =  \sum_{l=0}^{\infty}\sum_{m=0}^{\infty}
          \mathscr{B}_m^{(-l)}(n)\ x^{l}y^{m}. 
\end{align*}

Thus we complete the proof of \eqref{e-2-3}.
\end{proof}

From \eqref{e-2-2}, we immediately obtain the following.

\begin{corollary}\label{Cor-3}
$$
\sum_{l=0}^{\infty}\sum_{m=0}^{\infty}\sum_{n=0}^{\infty}\mathscr{B}_m^{(-l)}(n)\frac{x^{l}}{l!}\frac{y^{m}}{m!}\frac{z^{n}}{n!}
= \frac{e^{x+y}}{e^{x}+e^{y}-e^{x+y}-z}.
$$
\end{corollary}

\ 

%%%%%%%%%%%%%%%%%%%%%%%%%%%%%%%%%%%%%%%%
\section{An analytic proof of the duality formula for $\mathscr{B}_m^{(-l)}(n)$}\label{sec-3}
%%%%%%%%%%%%%%%%%%%%%%%%%%%%%%%%%%%%%%%%

In this section, we give an analytic proof of the duality formula (Corollary \ref{Cor-1}) 
for $\mathscr{B}_m^{(-l)}(n)$  by using certain zeta-function.

Arakawa and the first-named author \cite{AK1999} defined the zeta-function 
\begin{equation*}
\xi_k(s)=\frac{1}{\Gamma(s)}\int_0^\infty {t^{s-1}}\frac{{\rm Li}_{k}(1-e^{-t})}{e^t-1}dt \quad ({\rm Re}(s)>0) 
\end{equation*}
for $k \in {\mathbb{Z}}_{\geq 1}$, 
which can be continued to $\mathbb{C}$ as an entire function. In particular, $\xi_1(s)=s\zeta(s+1)$. It is known that
\begin{equation*}
\xi_k(-m)=(-1)^m C_m^{(k)}  
\end{equation*}
for $m\in \mathbb{Z}_{\geq 0}$ (see \cite[Theorem 6]{AK1999}). Note that they further study a multiple version of $\xi_k(s)$. 
\if0
As a generalization of $\xi(k;s)$, Coppo and Candelpergher \cite{CC2010} defined 
\begin{equation}
\xi_k(s;w)=\frac{1}{\Gamma(s)}\int_0^\infty {t^{s-1}}e^{-wt}\frac{{\rm Li}_{k}(1-e^{-t})}{1-e^{-t}}dt  \label{CC}
\end{equation}
for $k\in \mathbb{Z}_{\geq 1}$ and $w>0$, and studied its property. Note that $\xi_k(s;1)=\xi_k(s)$. 
\fi

Recently the first-named and the third-named author \cite{Kaneko-Tsumura2015} defined another type of 
Arakawa-Kaneko zeta-function by 
\begin{equation*}
\eta_k(s)=\frac{1}{\Gamma(s)}\int_{0}^\infty t^{s-1}\frac{{\rm Li}_{k}(1-e^t)}{1-e^t}dt 
\end{equation*}
for $s\in \mathbb{C}$ and $k\in \mathbb{Z}$, which interpolates the poly-Bernoulli numbers of $B$-type, that is,
\begin{equation}
\eta_k(-m)=B_m^{(k)}\qquad (m\in \mathbb{Z}_{\geq 0}). \label{eta-val}
\end{equation}
We emphasize that $\eta_k(s)$ is defined for any $k\in \mathbb{Z}$ while $\xi_k(s)$ is defined for $k\in \mathbb{Z}_{\geq 1}$. 
In fact, investigating $\eta_{-k}(s)$ $(k\in \mathbb{Z}_{\geq 0})$, they gave an alternative proof of \eqref{e-1-5} 
(the case $r=1$ of \cite[Theorem 4.7]{Kaneko-Tsumura2015}). 

\if0
Also, investigating 
\begin{equation}
\widetilde{\xi}_{-k}(s)=\frac{1}{\Gamma(s)}\int_{0}^\infty t^{s-1}\frac{{\rm Li}_{-k}(1-e^t)}{e^{-t}-1}dt\quad 
(k\in \mathbb{Z}_{\geq 0}), \label{tilde-xi}
\end{equation}
they recovered \eqref{e-1-5-2}.
\fi

Here we briefly recall this technique (for the details, see \cite[Section 4]{Kaneko-Tsumura2015}), and consider its generalization as follows. 
Let
\begin{equation*}
\mathcal{G}(u,t): =\frac{e^{u}}{1-e^u(1-e^t)} 
\end{equation*}
and 
\begin{equation}
\mathcal{F}(u,s):=\frac{1}{\Gamma(s)(e^{2\pi i s}-1)}\int_\mathcal{C} t^{s-1}\mathcal{G}(u,t)dt, \label{eq-3-1}
\end{equation}
where $\mathcal{C}$ is the well-known contour, namely the path consisting of the positive real axis (top side), 
a circle $C_\varepsilon$ around the origin of radius $\varepsilon$ (which is sufficiently small), and the positive 
real axis (bottom side) (see, for example, \cite[Theorem 4.2]{Wa}). 
We can write the integral as
\begin{align}
\mathcal{F}(u,s)& =\frac{1}{\Gamma(s)}\int_{\varepsilon}^\infty t^{s-1}\mathcal{G}(u,t)\,dt+
\frac{1}{\Gamma(s)(e^{2\pi i s}-1)}\int_{C_\varepsilon}t^{s-1}\mathcal{G}(u,t)\,dt.\label{cont}
\end{align}
Suppose ${\rm Re}(s)>0$ and let $\varepsilon \to 0$. Then
\begin{align}
\mathcal{F}(u,s)&=\frac{1}{\Gamma(s)}\int_{0}^\infty t^{s-1}\mathcal{G}(u,t)dt=\sum_{m=0}^\infty \eta_{-m}(s)\frac{u^m}{m!},  \label{eq-2}
\end{align}
because 
\begin{equation}
\mathcal{G}(u,t)=\frac{1}{1-e^t}\sum_{m=0}^\infty{\rm Li}_{-m}(1-e^t)\frac{u^{m}}{m!} \label{eq-1}
\end{equation}
(see \cite[Lemma 5.9]{Kaneko-Tsumura2015}). 
We also see that 
\begin{align}
\mathcal{G}(u,t)& =\frac{e^{u}}{1-e^u(1-e^t)}= \frac{e^{-t}}{1-e^{-t}(1-e^{-u})}=\sum_{l=1}^\infty e^{-lt}(1-e^{-u})^{l-1}. \label{eq-5}
\end{align}
Substituting \eqref{eq-5} into the second member of \eqref{eq-2}, we have
\begin{equation}
\mathcal{F}(u,s)=\frac{{\rm Li}_{s}(1-e^{-u})}{1-e^{-u}}=\sum_{m=0}^\infty B_m^{(s)}\frac{u^m}{m!}, \label{eq-6}
\end{equation}
where we define ${\rm Li}_s(z)$ and $\{ B_m^{(s)}\}_{m\geq 0}$ by replacing $k$ by $s\in \mathbb{C}$ in \eqref{e-1-1} and \eqref{e-1-3}, respectively. 
Comparing \eqref{eq-2} and \eqref{eq-6}, we have
\begin{equation}
\eta_{-m}(s)=B_m^{(s)}. \label{Dual-func}
\end{equation}
Letting $s=-k\in \mathbb{Z}_{\leq 0}$ in \eqref{Dual-func} and using \eqref{eta-val} and \eqref{cont}, we obtain $B_k^{(-m)}=B_m^{(-k)}$.

Next we generalize this result. Let
\begin{equation*}
\mathcal{G}_n(u,t): =e^{nt}\sum_{j=0}^{n} {n \brack j} \frac{\partial^j}{\partial u^j}\mathcal{G}(u,t)\quad (n\in \mathbb{Z}_{\geq 0}). 
\end{equation*}
Note that $\mathcal{G}_0(u,t)=\mathcal{G}(u,t)$. We prove the following.

\begin{lemma}\label{L-Gk}\ \ 
For $n\in \mathbb{Z}_{\geq 0}$,
\begin{align}
\mathcal{G}_n(u,t)& =e^{-nu}\sum_{m=1}^\infty\,\frac{(m+n-1)!}{(m-1)!}e^{-mt}(1-e^{-u})^{m-1}. \label{eq-2-3}
\end{align}
\end{lemma}

\begin{proof}\ \ We give the proof by induction on $n$. As for $n=0$, \eqref{eq-2-3} coincides with \eqref{eq-5}. 

Using \eqref{Stir}, we can check that
$$\frac{\partial}{\partial u}\mathcal{G}_n(u,t)=e^{-t}\mathcal{G}_{n+1}(u,t)-n\mathcal{G}_n(u,t).$$
Hence we have
\begin{align*}
\mathcal{G}_{n+1}(u,t)&=e^t\left(\frac{\partial}{\partial u}\mathcal{G}_n(u,t)+n\mathcal{G}_n(u,t)\right)\\
& =e^t\left(e^{-nu}\sum_{m= 2}^\infty \frac{(m+n-1)!}{(m-2)!}e^{-mt}(1-e^{-u})^{m-2}e^{-u}\right).
\end{align*}
Replacing $m$ by $m+1$, we have the assertion.
\end{proof}

Similar to \eqref{eq-3-1}, let 
\begin{align}
\mathcal{F}_n(u,s):&=\frac{1}{\Gamma(s)(e^{2\pi i s}-1)}\int_\mathcal{C} t^{s-1}\mathcal{G}_n(u,t)dt\notag\\
&=\frac{1}{\Gamma(s)}\int_{\varepsilon}^\infty t^{s-1}\mathcal{G}_n(u,t)dt+\frac{1}{\Gamma(s)(e^{2\pi i s}-1)}\int_{C_\varepsilon} t^{s-1}\mathcal{G}_n(u,t)dt.  \label{eq-9}
\end{align}

Assume $n\geq 1$. 
First, for ${\rm Re}(s)>1$, let $\varepsilon \to 0$ in \eqref{eq-9}. Then we obtain from \eqref{eq-2-3} that
\begin{align}
\mathcal{F}_n(u;s)&=\frac{e^{-nu}}{\Gamma(s)}\sum_{m=1}^\infty\,\frac{(m+n-1)!}{(m-1)!}(1-e^{-u})^{m-1}\int_0^\infty t^{s-1}e^{-mt}dt \notag\\
&=\frac{e^{-nu}}{1-e^{-u}}\sum_{m=1}^\infty\,\frac{(m+n-1)\cdots (m+1)m}{m^s}(1-e^{-u})^{m}\notag\\
& =\sum_{j=0}^{n}{n \brack j}e^{-nu}\frac{{\rm Li}_{s-j}(1-e^{-u})}{1-e^{-u}}\notag\\
& =\sum_{m=0}^\infty\,\sum_{j=0}^{n}{n \brack j} B_m^{(s-j)}(n)\frac{u^m}{m!}. \label{eq-3-2}
\end{align}

Secondly, by \eqref{eq-1}, we have
$$\frac{\partial^j}{\partial u^j}\mathcal{G}(u,t)=\frac{1}{1-e^t}\sum_{m=0}^\infty{\rm Li}_{-m-j}(1-e^t)\frac{u^{m}}{m!}.$$
Hence we obtain from \eqref{e-1-6} that
\begin{align*}
\mathcal{G}_n(u,t)
& =\sum_{m=0}^\infty \sum_{j=0}^{n}{n \brack j} e^{nt}\frac{{\rm Li}_{-m-j}(1-e^t)}{1-e^{t}}\frac{u^m}{m!}\\
&=\sum_{m=0}^\infty \sum_{k=0}^\infty \sum_{j=0}^{n}{n \brack j} B_k^{(-m-j)}(n)\frac{(-t)^k}{k!}\frac{u^m}{m!}. 
\end{align*}
Hence, letting $s\to -l$ for $l\in \mathbb{Z}_{\geq 0}$ in \eqref{eq-9}, 
we have
\begin{align}
\mathcal{F}_n(u;-l)&=\lim_{s\to -l}\frac{1}{\Gamma(s)(e^{2\pi i s}-1)}\int_{C_\varepsilon} t^{-l-1}\mathcal{G}_n(u,t)dt \notag\\
&=\sum_{m=0}^\infty\sum_{j=0}^{n}{n \brack j} B_l^{(-m-j)}(n)\,\frac{u^m}{m!}. \label{e-3-3-2}
\end{align}
Comparing the coefficients of \eqref{eq-3-2} with $s=-l$ and \eqref{e-3-3-2}, 
we obtain the proof of Corollary \ref{Cor-1}.

\begin{remark}
As a continuation of the observation stated in \cite[Section 4]{Kaneko-Tsumura2015}, we first found 
the duality formula \eqref{e-2-4} by the method described in this section. 
And then we gave its elementary proof presented in Section \ref{sec-2}.
\end{remark}

%%%%%%%%%%%%%%%%%%%%%%%%%%%%%%%%%%%%%
\section{A formula relating poly-Bernoulli numbers with Genocchi numbers}\label{sec-4}
%%%%%%%%%%%%%%%%%%%%%%%%%%%%%%%%%%%%%

In this section, we prove the $C$-type version of the following known result for $B_m^{(-l)}$:

\begin{prop}[\cite{AK-R1999} Proposition]\label{Prop4-1}\ \ For any $n\in \mathbb{Z}_{\geq 1}$, we have
\begin{equation*}
\sum_{l=0}^n (-1)^lB_{n-l}^{(-l)}=0. 
\end{equation*}
\end{prop} 

If we consider the $C$-version of the left-hand side of this identity, the value is not 0
but turns out to be the Genocchi number.
The Genocchi numbers $\{{G}_{n}\}_{n\geq 0}$ are  defined by the generating series
$$\frac{2t}{e^t+1}=\sum_{n=0}^\infty G_n\frac{t^n}{n!}.$$
(See, for example, Lucas \cite[P.\,250]{Lucas}, also Stanley \cite[Exercise 5.8]{Stanley}). 
Note that the relation with Bernoulli numbers 
$$G_{n} = (2-2^{n+1})B_{n}\quad (n\in \mathbb{Z}_{\geq 0})$$
holds and $G_n$ is an integer for all $n$. The first several values of $G_n$ are 
$$ 0, \,1,\, -1,\, 0,\, 1,\, 0,\, -3,\, 0,\, 17,\, 0,\, -155,\, 0,\, \ldots.$$
The second main result of this paper is the following.

\begin{theorem}\label{Th-main-2}\ \ For any $n\in \mathbb{Z}_{\geq 0}$, we have
\begin{equation}
\sum_{l=0}^{n}(-1)^{l}C_{n-l}^{(-l-1)} = -G_{n+2}.\label{e-4-0}
\end{equation}
\end{theorem}

\begin{remark}
We may write the identity as
$$ \sum_{l=0}^{n}(-1)^{l}C_{n-l}^{(-l)} = G_{n+1},$$
because $C_n^{(0)}=0$ for $n\ge 1$ and $C_0^{(0)}=1$.  However,
because of the duality \eqref{e-1-5-2}, we state and prove the identity
as given in the theorem.
\end{remark}

The rest of this section is devoted to the proof of Theorem \ref{Th-main-2}. 

The generating function of the left-hand side of \eqref{e-4-0}, which we denote by
$f(x)$, is obtained from  \eqref{e-2-3} by specializing $n=1$ and $y=-x$:
\begin{align*}
f(x)
& =  \sum_{m=0}^{\infty}\sum_{l=0}^{\infty}\mathscr{B}_m^{(-l)}(1)x^{m}(-1)^{l}x^{l}  =  \sum_{n=0}^{\infty}\sum_{l=0}^{n}(-1)^{l}C_{n-l}^{(-l-1)}x^{n}.
\end{align*}
Let $g(x)$ be the generating function of the sequence $\{-G_n\}_{n=0}^\infty$:
\begin{equation*}
g(x)  = - \sum_{n=0}^{\infty}G_{n}x^{n}=\sum_{n=0}^{\infty}(2^{n+1}-2)B_{n}x^{n}=-x+x^2-x^4+3x^6-17x^8+\cdots. 
\end{equation*}
Then, our assertion  \eqref{e-4-0} can be rewritten as 
\begin{equation*}
g(x) = x^{2}f(x)-x. 
\end{equation*}
It is convenient for our purpose to make a shift and define
\begin{equation*}
f_1(x)=x^2f(x),\quad g_1(x)=xg(x).  
\end{equation*}
With these, our goal is to prove the identity
\begin{equation*}
g_1(x) = xf_1(x)-x^2. %\label{e-4-3}
\end{equation*}
To show this, we proceed as follows.  We first show that the power series $g_1(x)$ is 
a unique element of $x\,\mathbb{Q}[[x]]$ satisfying the functional equation
\begin{equation}
g_1\left(\frac{x}{1-2x}\right)=g_1(x)+\frac{2x^{3}(x-2)}{(1-x)^{2}},  \label{e-4-1}
\end{equation}
and then show that the right-hand side $xf_1(x)-x^2$ also satisfies the 
same functional equation, thereby proving the theorem by the uniqueness.

The first step is carried out in a similar manner as in the proof of the following proposition of Don Zagier.

\begin{prop}[\cite{AIK2014},\,Proposition A.1 in Appendix]\label{Zagier} \ 
\ The power series 
$$\beta_1(x) = \sum_{n=0}^{\infty}B_{n}x^{n+1}$$
is the unique solution in $x\,\mathbb{Q}[[x]]$ of the equation
\begin{equation}
\beta_{1} \left( \frac{x}{1-x} \right) = \beta_{1}(x) +x^{2}. \label{e-4-2}
\end{equation}
\end{prop}
Since $g_1(x)=\sum_{n=0}^{\infty}(2^{n+1}-2)B_{n}x^{n+1}=\beta_1(2x)-2\beta_1(x)$, the identity \eqref{e-4-1} is easily 
derived from \eqref{e-4-2} by replacing $x$ by $x/(1-x)$ and applying \eqref{e-4-2} 
again.  The proof of the uniqueness, which we state as the lemma below, is postponed
to the end of this section.

\begin{lemma}\label{lemma-unique}
Let
$$
h(x) = \sum_{n=0}^{\infty}d_{n}x^{n+1}\in x\,\mathbb{Q}[[x]]
$$
satisfies \eqref{e-4-1}, i.e., 
\begin{equation}
h\left(\frac{x}{1-2x}\right)=h(x)+\frac{2x^{3}(x-2)}{(1-x)^{2}},  \label{eq-4-10}
\end{equation}
then we have
\begin{equation}
d_{n} = (2^{n+1}-2)B_{n}\quad (n\in \mathbb{Z}_{\geq 0}). \label{eq-4-8-2}
\end{equation}
\end{lemma}

Now we are going to prove the series $f_2(x):=xf_1(x)-x^2$ satisfies the same
functional equation
\begin{equation}
f_2\left(\frac{x}{1-2x}\right)=f_2(x)+\frac{2x^{3}(x-2)}{(1-x)^{2}}.  \label{e-4-5}
\end{equation}
By \eqref{e-2-3} ($n=1$ and $y=-x$), we have 
\begin{equation*}
f_{1}(x) = x^{2}f(x)=\sum_{j=0}^{\infty} 
\frac{(-1)^{j}j!\ (j+1)!\ x^{2j+2}}{\prod_{\nu=1}^{j+1}(1-\nu x)(1+\nu x)}. 
\end{equation*}
Let $a_j(x)$ be the $j$th term in the sum on the right, 
$$a_j(x)=\frac{(-1)^{j}j!\ (j+1)!\ x^{2j+2}}{\prod_{\nu=1}^{j+1}(1-\nu x)(1+\nu x)},$$
so that $f_1(x)=\sum_{j=0}^\infty a_j(x)$. 
A simple calculation shows that the functional equation \eqref{e-4-5} is equivalent to the 
functional equation 
\begin{equation*}
f_{1} \left( \frac{x}{1-2x} \right) =(1-2x)f_{1}(x)+\frac{2x^{3}(3-6x+2x^{2})}{(1-x)^{2}(1-2x)} %\label{eq-4-6}
\end{equation*}
for $f_1(x)$.  This follows then from the next lemma, because the right-hand side of \eqref{eq-4-7}
is  in $x^{2n+5} \mathbb{Q}[[x]]$ and $n$ can be arbitrary large.

\begin{lemma}\label{lemma-key}\ \ For any $n\in \mathbb{Z}_{\geq 0}$, we have
\begin{align} 
& \sum_{j=0}^{n}\left(a_{j}\left(\frac{x}{1-2x}\right)-(1-2x)a_{j}(x)\right) -\frac{2x^{3}(3-6x+2x^{2})}{(1-x)^{2}(1-2x)}  \notag\\
&= -\frac{2x}{1-x}\cdot\frac{1+(n+2)x}{1-(n+3)x}\cdot \left((n+3)(x-1)^2-(n+2)(2x-1)\right) a_{n+1}(x). \label{eq-4-7}
\end{align}
\end{lemma}

\begin{proof}\ 
The proof is by induction on $n\geq 0$, and is a straightforward calculation which we omit.
\end{proof}

%We leave 
%As for $l=0$, we have
%\begin{align*}
%& a_{0}(x)-b_{0}(x)-c(x) \notag\\
%& =  \frac{x^2\{ (1+x)-(1-2x)(1-3x)\}}{(1-x)(1-3x)(1+x)}
%          -\frac{2x^{3}(3-6x+2x^{2})}{(1-x)^{2}(1-2x)} \\
%& =  \frac{4x^5(5-10x+3x^2)}{(1-x)^2(1-2x)(1-3x)(1+x)},
%\end{align*}
%which implies \eqref{eq-4-7} for $l=0$. 
%
%Next we assume \eqref{eq-4-7} for the case of $l$. We have
%$$
%a_{l+1}(x) = \frac{(-1)^{l+1}(l+1)!\ (l+2)!\ x^{2l+4}}
%{(1-x)(1-3x) \cdots (1-(l+4)x)(1+x)(1+2x) \cdots (1+lx)},
%$$
%$$
%b_{l+1}(x) =  \frac{ (-1)^{l+1}(l+1)!\ (l+2)!\ x^{2l+4}}
%{(1-x)(1-3x) \cdots (1-(l+2)x)(1+x)(1+2x) \cdots (1+(l+2)x)}.
%$$
%Hence we obtain
%\begin{align*}
%& \sum_{j=0}^{l+1}(a_{j}(x)-b_{j}(x))  -c(x) \\
%& =   \sum_{j=0}^{l}(a_{j}(x)-b_{j}(x))  
%-c(x)+a_{l+1}(x)-b_{l+1}(x) \\
%& =  \frac{(-1)^{l}2(l+1)!\ (l+2)!\ x^{2l+5}\{ (2l+5)-2(2l+5)x+(l+3)x^{2}\}}{(1-x)^{2}(1-2x)\cdots(1-(l+3)x)(1+x)(1+2x) \cdots (1+(l+1)x)} \\
%&\quad +\frac{(-1)^{l+1}(l+1)!\ (l+2)!\ x^{2l+4}\left \{2x(2l+5)(1-x) \right \}}{(1-x)(1-3x) \cdots (1-(l+4)x)(1+x)(1+2x) \cdots (1+(l+2)x)} \\
%& =  \frac{(-1)^{l+1}2(l+2)!\ (l+3)!\ x^{2l+7}\{ (2l+7)-2(2l+7)x+(l+4)x^{2}\}}{(1-x)^{2}(1-2x)\cdots(1-(l+4)x)(1+x)(1+2x) \cdots (1+(l+2)x)},
%\end{align*}
%which implies \eqref{eq-4-7} for the case of $l+1$. This completes the proof.

\begin{proof}[Proofs of Lemma \ref{lemma-unique} and Theorem \ref{Th-main-2}]\ 
Because of the binomial expansion
$$
(1-2x)^{-n-1} = \sum_{j=0}^{\infty}{n+j \choose j}2^{j}x^{j},
$$
the left-hand side of \eqref{eq-4-10} is equal to
\begin{align*}
& \sum_{n=0}^{\infty}d_{n}\sum_{j=0}^{\infty}{n+j \choose j}2^{j}x^{n+j+1}=  \sum_{m=0}^{\infty}\left( \sum_{n=0}^{m}{m \choose n}2^{m-n}d_{n}\right) x^{m+1}.
\end{align*}
On the other hand, since
$$\frac{2x^3(x-2)}{(1-x)^2} =-2+2x^2 +\frac{4}{1-x}-\frac{2}{(1-x)^2} =-2 \sum_{m=2}^{\infty} mx^{m+1},$$
the right-hand side of \eqref{eq-4-10} is equal to
$$\sum_{m=0}^{\infty}d_{m}x^{m+1} 
- 2\sum_{m=2}^{\infty}mx^{m+1}.
$$
Comparing the coefficients, we have $d_0=0$ and 
\begin{equation*}
\sum_{n=0}^{m-1}{m \choose n}2^{m-n}d_{n} = -2m\ \ \ \ (m \geq 2). \label{eq-4-11}
\end{equation*}
Therefore, since this recursion (with $d_0=0$)  uniquely determines the 
numbers $d_n\ (n\ge1)$,  we only need to prove 
\begin{equation}
\sum_{n=0}^{m-1}{m \choose n}2^{m-n}(2^{n+1}-2)B_{n} = -2m\ \ \ \ (m \geq 2) \label{eq-4-12}
\end{equation}
in order to establish \eqref{eq-4-8-2}.
By using the standard recursion
$$
\sum_{n=0}^{m-1}{m \choose n}B_{n}=0\ \ \ \ (m \geq 2),
$$
we can rewrite \eqref{eq-4-12} as
\begin{equation*}
\sum_{n=0}^{m}{m \choose n}2^{m-n}B_{n} = m+B_{m}\ \ \ \ (m \geq 2). %\label{eq-4-13}
\end{equation*}
This can be easily verified by manipulating the generating function:
\begin{align*}
& \frac{x}{e^x -1}\cdot e^{2x} =  \sum_{m=0}^{\infty} \left( \sum_{n=0}^{m}{m \choose n}2^{m-n}B_{n} \right)\frac{x^m}{m!}
\end{align*}
and 
\begin{equation*}
\frac{x}{e^x -1}\cdot e^{2x} %& =  \frac{x(e^{2x}-1)+x}{e^x - 1} \\ 
 =  x(e^x + 1) + \frac{x}{e^x - 1}  
 =  \sum_{m=1}^{\infty}m \cdot \frac{x^{m}}{m!} + \sum_{m=0}^{\infty}B_{m}\frac{x^m}{m!} +x.
\end{equation*}
This completes the proof of Lemma \ref{lemma-unique}, and thus Theorem \ref{Th-main-2} is proved.
\end{proof}

\ 

\ 

\begin{flushleft}
\begin{small}
{M. Kaneko}: 
{Faculty of Mathematics, 
Kyushu University, 
Motooka 744, Nishi-ku
Fukuoka 819-0395, 
Japan}

e-mail: {mkaneko@math.kyushu-u.ac.jp}

\ 

{F. Sakurai}: 
{Graduate School of Mathematics, 
Kyushu University, 
Motooka 744, Nishi-ku
Fukuoka 819-0395, 
Japan}

e-mail: {f-sakurai@kyudai.jp}

\

{H. Tsumura}: 
{Department of Mathematics and Information Sciences, Tokyo Metropolitan University, 1-1, Minami-Ohsawa, Hachioji, Tokyo 192-0397 Japan}

e-mail: {tsumura@tmu.ac.jp}
\end{small}
\end{flushleft}

\end{document}